\documentclass[11pt,a4paper]{article}

\usepackage[margin=1in]{geometry}
\usepackage[dvipsnames]{xcolor}
\usepackage[utf8]{inputenc}
  \usepackage[T1]{fontenc}
  \usepackage{textcomp}
  \usepackage{lmodern}
\usepackage{microtype}
\usepackage{amsmath,amssymb,amsthm}
\usepackage{mathtools}
\usepackage{thmtools}
\makeatletter
\@ifundefined{newcounteralias}{}{%
  \renewcommand\thmt@autorefsetup{%
    \@xa\def\csname\thmt@envname autorefname\@xa\endcsname\@xa{\thmt@thmname}%
  }%
}
\makeatother
\usepackage[numbers]{natbib}
\usepackage[colorlinks=true,linkcolor=BrickRed,citecolor=Green,urlcolor=Blue]{hyperref}
\usepackage[capitalise,noabbrev]{cleveref}
\crefname{conjecture}{Conjecture}{Conjectures}
\crefname{problem}{Problem}{Problems}
\makeatletter
\crefname{ALG@line}{Line}{Lines}
\makeatother
\usepackage{soul}
\usepackage{graphicx}
\usepackage{subcaption}
\usepackage{tabularx}
\usepackage{booktabs}
\usepackage{array}
\usepackage{calc}
\usepackage{algorithm}
\usepackage{algpseudocode}
\usepackage{etoolbox}
\makeatletter
\patchcmd{\ALG@step}{\addtocounter{ALG@line}{1}}{\refstepcounter{ALG@line}}{}{\GenericError{}{Failed to patch \noexpand\ALG@step}{}{}}
\newcommand{\linelabel}[1]{\begingroup
  \cref@old@label{#1}%
  \protected@edef\cref@currentlabel{[ALG@line][\arabic{ALG@line}][]\theALG@line}%
  \let\cref@old@label\@gobble
  \label{#1}%
\endgroup}
\makeatother
\usepackage{listings}
\newtheorem{theorem}{Theorem}
\newtheorem{lemma}[theorem]{Lemma}

\newtheorem{corollary}[theorem]{Corollary}
\newtheorem{conjecture}[theorem]{Conjecture}

\NewDocumentCommand\CSLBibitem{o m}{}
\newlength{\cslhangindent}
\newlength{\csllabelwidth}
\newlength{\cslentryspacingunit}
\DeclarePairedDelimiter{\set}{\{}{\}}
\DeclarePairedDelimiterX{\Set}[2]{\{}{\}}{\,#1:#2\,}

\DeclareMathOperator{\cl}{cl}
\DeclareMathOperator{\stab}{stab}
\DeclareMathOperator*{\argmin}{arg\,min}
\DeclareMathOperator*{\argmax}{arg\,max}

\newcommand{\Z}{\mathbb{Z}}

\newcommand{\R}{\mathbb{R}}

\newcommand{\symdiff}{\mathbin{\triangle}}

\title{Proximity Theorem and FPT Algorithm for\\ Minimum Group-labeled Matroid Bases}

\author{
    Siyue Liu\thanks{Carnegie Mellon University, Pittsburgh, USA. E-mail: \texttt{siyueliu@andrew.cmu.edu}.}
    \and Taihei Oki\thanks{Institute for Chemical Reaction Design and Discovery (ICReDD), Hokkaido University, Sapporo, Hokkaido, Japan. D3 Center, The University of Osaka, Osaka, Japan. Center for Advanced Intelligence Project, RIKEN, Tokyo, Japan. E-mail: \texttt{oki@icredd.hokudai.ac.jp}.}
    \and Chao Xu\thanks{University of Electronic Science and Technology of China, Chengdu, China. E-mail: \texttt{the.chao.xu@gmail.com}.}
}
\date{}

\begin{document}

\maketitle

\begin{abstract}
Let \(M\) be a matroid, let \(\ell:E(M)\to G\) be a labeling of its elements by an abelian group, and let \(F\subseteq G\) be finite. A base is \emph{\(F\)-avoiding} if the label sum of its elements is not in $F$. We prove that, for arbitrary costs, a min-cost \(F\)-avoiding base can be obtained from any min-cost base using at most \(|F|\) exchanges. This resolves the weighted proximity conjecture by Hörsch, Imolay, Mizutani, Oki, and Schwarcz. Our main technical tool is to replace the assumption of Schrijver--Seymour Conjecture in Liu and Xu's earlier work with Kemperman--Scherk Theorem. The proximity theorem combined with approaches in earlier work gives an FPT algorithm that finds a min-cost \(F\)-avoiding base parameterized by \(|G|\). We additionally give an FPT algorithm parameterized by \(|F|\), answering a question of Hörsch et al.
\end{abstract}

\section{Introduction}\label{sec:introduction}

Let \(M\) be a finite matroid on ground set \(E\), let \( G\) be an abelian (additive) group, let \(\ell:E\to G\) label the elements, and write \(\ell(B)\coloneqq \sum_{e\in B}\ell(e)\) for a subset $B\subseteq E$. For a finite set \(F\subseteq G\), a base \(B\) is \emph{\(F\)-avoiding} if \(\ell(B)\notin F\). Constraining \(\ell(B)\) encodes parity, congruency, exact-target, and forbidden-value side constraints, which have been studied for matchings~\citep{papadimitriouYannakakis1982Restricted, mulmuleyVaziraniVazirani1987Matching}, arborescences~\citep{barahonaPulleyblank1987Exact}, submodular function minimization~\citep{goemansRamakrishnan1995Minimizing, nageleSudakovZenklusen2019Submodular}, minimum cuts~\citep{nageleZenklusen2020Contraction}, paths in group-labeled graphs~\citep{chudnovskyEtAl2006Packing}, and matroid bases~\citep{rieder1991Lattices, cameriniGalbiatiMaffioli1992Exact, doronAradKulikShachnai2024Lower, horschEtAl2024GroupLabeled, garamvolgyiEtAl2025Proximity}. We study the min-cost $F$-avoiding bases: additionally given a cost function over the ground set \(c:E\to\R\),  we are interested in finding a min-cost $F$-avoiding base, where the cost of a base $B$ is $c(B)\coloneqq \sum_{e\in B}c(e)$. Here, a min-cost $F$-avoiding base refers to one with the smallest cost among all $F$-avoiding bases; it does not have to have the smallest cost among all bases. The \emph{distance} between bases \(A\) and \(B\) is \(d(A,B)\coloneqq |A\setminus B|\), which equals the number of basis exchanges needed to reach one from the other. We ask how far an unconstrained optimum is from a constrained one: given an arbitrary min-cost base, how far is the closest base that is of minimum cost among the \(F\)-avoiding bases?

The unweighted version with costs $c\equiv 0$ has received more attention and substantial progress has been made recently. Namely, given an arbitrary base, how far is the closest $F$-avoiding base, provided that an \(F\)-avoiding base exists? For \(g\in G\), a base \(B\) is a \emph{\(g\)-base} if \(\ell(B)=g\). Liu and Xu~\citep{liuXu2024Congruency} conjectured that, for every finite abelian group \( G\), every attainable target \(g\), and every prescribed base, the closest \(g\)-base to it has distance at most \(| G|-1\). This is a special case of the $F$-avoiding setting by putting $F=G\setminus g$. Liu and Xu proved the conjecture for cyclic groups whose size is a prime-power, as well as for groups whose size is a product of two primes. More interestingly, they reduce the conjecture to an old conjecture of Schrijver and Seymour \cite{schrijver1990spanning} on counting labels of bases. H\"orsch, Imolay, Mizutani, Oki, and Schwarcz~\citep{horschEtAl2024GroupLabeled} proposed the \emph{proximity conjecture} for general $F$-avoiding bases, which states that every base has an \(F\)-avoiding base within distance \(|F|\), provided that an $F$-avoiding base exists.
For \(|F|=1\), the conjecture essentially follows from Rieder's characterization of base lattices~\citep{rieder1991Lattices}. H\"orsch et al.~\citep{horschEtAl2024GroupLabeled} proved the conjecture when \(|F|\le 2\), when the group is ordered, and for strongly base-orderable matroids.
Garamv\"olgyi, Mizutani, Oki, Schwarcz, and Yamaguchi~\citep{garamvolgyiEtAl2025Proximity} proved the conjecture when \(|F|\le4\), via a computer-assisted case analysis, and for sparse paving matroids. Most recently, Chen~\citep{chen2026Sharp} proved the conjecture for representable matroids, and derived a sharp proximity bound for the more general setting where multiple labelings are given, which was introduced in~\citep{garamvolgyiEtAl2025Proximity}.
Bhardwaj~\citep{bhardwaj2026Sharp} independently gave an alternative proof of the proximity conjecture for representable matroids.

Liu and Xu~\cite{liuXu2024Congruency} also proposed the weighted generalization of their conjecture with the same proximity bound, and H\"orsch et al.~\citep{horschEtAl2024GroupLabeled} further generalized it to the $F$-avoiding setting, which is known as the \emph{weighted proximity conjecture}. Namely, every min-cost base has a min-cost \(F\)-avoiding base within distance \(|F|\), provided that an $F$-avoiding base exists. Very little progress was made towards the weighted proximity conjecture compared with the unweighted one. Liu and Xu proved in the special case of $|F|=|G|-1$ a stronger bound \(D( G)-1\), where \(D( G)\) is the Davenport constant, for strongly base-orderable matroids, and for groups of size at most four. H\"orsch et al.~\citep{horschEtAl2024GroupLabeled}  proved the weighted conjecture for the special case where \(|F|=1\).

In this paper, we prove the weighted proximity conjecture in full, formally stated as follows.

\begin{restatable}[Proximity]{theorem}{proximity}\label{thm:proximity}
    Let \(M=(E,\mathcal{I})\) be a matroid, \( G\) be an abelian group, $F\subseteq G$ be a finite subset, $\ell:E\rightarrow G$ be a labeling, and $c:E\rightarrow \mathbb{R}$ be a cost function. For every min-cost base \(A\), there exists a min-cost \(F\)-avoiding base \(B\) such that $d(A,B)\le |F|$, provided that an \(F\)-avoiding base exists.
\end{restatable}

The bound is sharp: in the rank-\(r\) partition matroid with partition classes \(\{a_i,b_i\}\) and capacity $1$ for each part, labels \(\ell(a_i)=0\) and \(\ell(b_i)=1\) over \( G=\Z\), zero cost, and \(F=\{0,1,\ldots,k-1\}\) with \(k\le r\), every \(F\)-avoiding base contains at least \(k\) elements \(b_i\), so the base \(\{a_1,\ldots,a_r\}\) has distance at least \(k\) from every min-cost \(F\)-avoiding base. For a finite group \( G\) and an attainable \(g\), taking \(F= G\setminus\{g\}\) makes the \(F\)-avoiding bases exactly the \(g\)-bases, so the conjecture of Liu and Xu follows. For the unweighted \(g\)-base case and groups not covered by~\citep{liuXu2024Congruency}, the best previously known bounds were \(O(| G|^5)\) for cyclic \( G\) and \(| G|^{O(\log | G|)}\) in general, derived in~\citep{garamvolgyiEtAl2025Proximity} from the sensitivity theorem of Eisenbrand, Rohwedder, and W\k{e}grzycki~\citep{eisenbrandRohwedderWegrzycki2024Sensitivity}.

The proximity theorem alone yields an algorithm finding a min-cost \(F\)-avoiding base with \(n^{O(|F|)}\) independence-oracle queries, where \(n\coloneqq |E|\), by enumerating the bases within distance \(|F|\) of a min-cost base, and H\"orsch et al.~asked whether finding an \(F\)-avoiding base is fixed-parameter tractable in \(|F|\)~\citep{horschEtAl2024GroupLabeled}. We answer the question affirmatively even for the weighted case: an $F$-avoiding base can be found in \(O\bigl(|F|^2n^2+{|F|}^{O(|F|)}T_{\mathrm{MI}}\bigr)\) time, and a min-cost \(F\)-avoiding base can be found in \(2^{O(|F|^2)}\cdot O(n^2\,T_{\mathrm{WMI}})\) time, where \(T_{\mathrm{MI}}\) and \(T_{\mathrm{WMI}}\) denote the running times of unweighted and weighted matroid intersection, respectively, for two matroids over \(n\) elements. In this sense, fixed-parameter tractability is best possible: no algorithm, even randomized, finds an \(F\)-avoiding base in time polynomial in \(n\) and \(|F|\). Indeed, labeling each element by its integer weight and taking \(F=\{0,\ldots,W\}\setminus\{T\}\), where \(W\) is the total weight, encodes deciding whether a base of weight exactly \(T\) exists, which requires time superpolynomial in \(n\) and \(W\) in the independence-oracle model~\citep[Theorem~3]{doronAradKulikShachnai2024Lower}.

The analogous proximity conjecture for several simultaneous labelings, posed in~\citep{garamvolgyiEtAl2025Proximity} together with a lower bound of \(2^k-1\) for \(k\) labelings, remains open for arbitrary matroids.


\subsection{Techniques}\label{sec:techniques}

\paragraph{Proximity theorem.}
We use the terms ``color'' and ``label'' interchangeably. A \emph{color class} is the set of elements with the same label. When $G$ is the cyclic group $\mathbb{Z}_p$ for some prime $p$, the following theorem by Schrijver and Seymour~\cite{schrijver1990spanning} establishes a lower bound on the number of distinct colors of bases in a matroid, in terms of the rank sum of the color classes.
For a matroid $M = (E, \mathcal{I})$ and a labeling function $\ell\colon E \to G$, let $\ell(M)\coloneqq \Set{\ell(B)}{\text{$B$ is a base of $M$}}$.

\begin{theorem}[{Schrijver--Seymour~\cite{schrijver1990spanning}}]\label{thm:schrijvver_Seymour}
    Let $M=(E,\mathcal{I})$ be a matroid with rank function $r$ and let $\ell:E\to \mathbb{Z}_p$ for some prime number $p$. Then, $$|\ell(M)|\geq \min\left\{p,\ \sum_{g\in\mathbb{Z}_p} r(E_g)-r(E)+1\right\},$$ where $E_g\coloneqq \Set{e\in E}{\ell(e)=g}$ for $g \in \Z_p$. 
\end{theorem}
For general abelian groups, the same statement does not hold. Instead, Schrijver and Seymour proposed the following conjecture for general groups that generalizes \Cref{thm:schrijvver_Seymour}. For a subset $F\subseteq G$, the \emph{stabilizer} of $F$ is defined by $\stab(F):=\{g\in G: g+F=F\}$. It is not hard to see that $\stab(F)$ is a subgroup of $G$.
\begin{conjecture}[Schrijver--Seymour~\cite{schrijver1990spanning}, see also~\cite{devos2009generalization}]\label{conj:schrijver_seymour}
    Let $M=(E,\mathcal{I})$ be a matroid with rank function $r$ and let $\ell:E\to G$ for some finite abelian group $G$. Let $H=\stab(\ell(M))$ be the stabilizer of $\ell(M)$. Then, $$|\ell(M)|\geq |H|\cdot \min\left\{\left(\sum_{Q\in G/H} r(E(Q))-r(E)+1\right),\frac{|G|}{|H|}\right\},$$ where $E(Q)\coloneqq \Set{e\in E}{\ell(e)\in Q}$ for $Q  \in G/H$. 
\end{conjecture}
To the best of our knowledge, \Cref{conj:schrijver_seymour} is solved when $|G|$ is a product of two primes or $G=\Z_{p^n}$ for a prime $p$ and a positive integer $n$~\cite{devos2009generalization}. Liu and Xu~\cite{liuXu2024Congruency} prove that \Cref{conj:schrijver_seymour} implies the unweighted version of \Cref{thm:proximity} in the special case where $|F|=|G|-1$: for every base $A$, there exists a $g$-base $B$ such that $d(A,B)\le |G|-1$, if one exists. 
Thus, for finite groups and $F=G\setminus\{g\}$, one possible route to the unweighted version of \Cref{thm:proximity} is via the Schrijver--Seymour conjecture. 

We bypass the difficulty by proving an analogue of \Cref{thm:schrijvver_Seymour} for general groups, under the additional assumption that there is a base $B$ of $M$ such that its label is \emph{unique}. The main tool is the Kemperman--Scherk theorem from additive number theory (\Cref{thm:kemperman-scherk})~\citep{scherk1955Distinct, kemperman1956Complexes, Kemperman1960SmallSumsets}. In fact, we prove the following weighted generalization of this analogue, which bounds from below the number of distinct colors of the bases of small cost, in terms of the rank sum of color classes filtered by the cost of the smallest element of $B$.
For a matroid $M = (E, \mathcal{I})$, a labeling function $\ell\colon E \to G$, a cost function $c\colon E \to \R$, and $\theta \in \R$, let $\ell_{\le\theta}(M)\coloneqq \Set{\ell(B)}{\text{$B$ is a base of $M$ with $c(B) \le \theta$}}$.

\begin{restatable}{theorem}{weightedcounting}\label{thm:weighted-counting}
    Let $M=(E,\mathcal{I})$ be a matroid with rank function $r$, $G$ be an abelian group, $\ell:E\rightarrow G$ be a labeling, and $c:E\rightarrow \R$ be a cost function. Suppose that $B$ is the unique min-cost $\ell(B)$-base of $M$. Then, $$|\ell_{\le c(B)}(M)|\ge \sum_{h\in G} r(R_h)-r(E)+1,$$
    where $R_h\coloneqq \Set{e\in E}{\ell(e)=h\text{ and }c(e)\le \min_{b\in B}c(b)}$.
\end{restatable}

To see how this counting approach implies the weighted proximity \Cref{thm:proximity}, we follow ideas of~\cite{liuXu2024Congruency}. Pick a min-cost base $A$ and a min-cost $F$-avoiding base $B$ of $M$ such that $d(A,B)$ is minimized. We consider the matroid $M'=(M\mid_{A\cup B})/(A\cap B)$ obtained from deleting elements outside $A\cup B$ and contracting elements inside $A\cap B$. The matroid $M'$ is a \emph{block matroid}, which means its ground set can be partitioned into two bases. The crucial observation is that $B'\coloneqq B\setminus A$ is the \emph{unique} min-cost $F'$-avoiding base of $M'$ for the shifted label set $F'=F-\ell(A\cap B)$. 
Moreover, $A\setminus B$ is a min-cost base of $M'$. A refined form of \Cref{thm:weighted-counting} (\Cref{thm:weighted-counting-block}) in block matroids then shows that the bases of $M'$ of cost at most $c(B')$ attain at least $r(M')+1$ distinct labels. On the other hand, there should be at most $|F'|+1$ distinct labels, since every base of $M'$ of cost at most $c(B')$ other than $B'$ itself has label in $F'$ by the uniqueness. Hence $d(A,B)=|B'|=r(M')\le |F'|=|F|$.

\paragraph{FPT algorithms.}
For a subset $B\subseteq E$, its \emph{signature} is the vector $\sigma(B)={(\sigma_h(B))}_{h\in G}\in\Z_{\ge0}^G$, where $\sigma_h(B)\coloneqq|\Set{e\in B}{\ell(e)=h}|$.
The label of $B$ is determined by its signature as $\ell(B)=\sum_{h\in G}\sigma_h(B)\cdot h$. Therefore, the search for $g$-bases can be reduced to the search for signatures that give the target label $g$.
Moreover, for subsets $A,B\subseteq E$, one has $\|\sigma(A)-\sigma(B)\|_1\le |A\mathbin{\triangle}B|$.
In particular, if bases $A,B$ of $M$ satisfy $d(A,B)\le d$, then $\|\sigma(A)-\sigma(B)\|_1\le 2d$. Therefore, the closeness of bases from the proximity theorem implies the closeness of signatures. 
Assuming the proximity theorem, Liu and Xu~\cite{liuXu2024Congruency} use these ideas to obtain an FPT algorithm by searching for correct signatures of short distances.
Their approach easily extends to finding min-cost $F$-avoiding bases.

Given a finite $F\subseteq G$, a positive integer $d$, and a base $A$, we call a vector $b\in\Z_{\ge0}^G$ a \emph{$d$-close valid signature} for $A$ if the following conditions hold:
\begin{equation}\label{eq:candidate}
    \sum_{h\in G}b_h=r,\qquad
    \sum_{h\in G}b_h\cdot h\notin F,\qquad
    \|b-\sigma(A)\|_1\le 2d.
\end{equation}
Here, $r$ denotes the rank of $M$.
When $G$ is finite, assuming \Cref{thm:proximity} with proximity bound $d$ in place of $|F|$ gives an $O(|G|^{2d}T_{\mathrm{WMI}})$-time algorithm.
First, find a min-cost base $A$ by the greedy algorithm.
By proximity, it suffices to consider bases $B$ with $d(A,B)\le d$, and hence their $d$-close valid signatures.
For each valid signature $b$, the requirement $\sigma(B)=b$ can be encoded by a partition matroid whose parts are $E_g\coloneqq\Set{e\in E}{\ell(e)=g}$ and whose capacities are $b_g$.
A min-cost base with signature $b$ can therefore be found by weighted matroid intersection.
There are at most $|G|^{2d}$ valid signatures, giving the claimed running time.
Applying the bound $d=|F|$ from \Cref{thm:proximity} yields an $O(|G|^{2|F|}T_{\mathrm{WMI}})$-time algorithm, which is FPT parameterized by $|G|$, but not by $|F|$.

Our idea is to construct a finite set $\Lambda\subseteq G$ that contains the labels of all elements in $A\mathbin{\triangle}B$ for a min-cost base $A$ and some min-cost $F$-avoiding base $B$, with $|\Lambda|$ bounded by a function of $|F|$.
This guarantees $\sigma_g(A)=\sigma_g(B)$ for all $g\notin\Lambda$, so the number of relevant valid signatures depends on $|\Lambda|$ and $d$, rather than on $|G|$.

In the unweighted case, let $D\coloneqq\Set{g-h}{g,h\in F}$.
If some component of $M$ contains elements $e,e'$ with $\ell(e)-\ell(e')\notin D$, a circuit exchange yields bases $A$ and $A-e+e'$ whose labels differ by an element outside $D$, so at least one is $F$-avoiding.
Otherwise, shifting the labels in each component by the label of a fixed reference element, and translating $F$ accordingly, gives an equivalent instance whose element labels all lie in $D$.
Since $|D|\le |F|(|F|-1)+1$, taking $\Lambda=D$ and $d=|F|$ in the signature-enumeration subroutine gives an FPT algorithm (\cref{thm:fpt}).

The weighted case requires a different construction of $\Lambda$ because an $F$-avoiding base obtained by a single exchange from a min-cost base need not be min-cost.
After perturbing the cost function so that all relevant optima are unique, the algorithm recursively finds a min-cost $(F\setminus\{f\})$-avoiding base $B_f$ for each $f\in F$, whenever one exists.
It then constructs $K\coloneqq\bigcup_{f\in F}(A\mathbin{\triangle}B_f)$, whose size is $O(|F|^2)$ by \Cref{thm:proximity}.
For every $a\in A$, the algorithm constructs a label set $\Lambda$ consisting of all labels reachable from $\Set{\ell(e)}{e\in K \cup \set{a}}$ by at most $2|F|$ \emph{derivation steps}, each of which adds an element of $F-\ell(A)$ or its inverse.
Every such set $\Lambda$ has size $2^{O(|F|)}$.
For each such $\Lambda$, we apply signature enumeration with $d=|F|$.
Using Schrijver's exchange lemma for linking systems~\citep{schrijver1979Linking}, we show that one of these sets $\Lambda$ contains the labels of all elements in $A\mathbin{\triangle}B$ for the min-cost $F$-avoiding base $B$, thus yielding an FPT algorithm (\cref{thm:weighted-fpt}).

\subsection{Organization}\label{sec:organization}
The rest of this paper is organized as follows.
\cref{sec:preliminaries} provides preliminaries.
\cref{sec:proximity} proves Theorem~\ref{thm:proximity}.
Finally, \cref{sec:fpt} presents the FPT algorithms for finding $F$-avoiding bases.

\section{Preliminaries}\label{sec:preliminaries}


Let \(G\) be an abelian group written additively, let \(\ell\colon E\to G\) be a labeling function, and let $c\colon E\to\R$ be a cost function.
For \(X\subseteq E\), write \(\ell(X)\coloneqq \sum_{e\in X}\ell(e)\) and \(c(X)\coloneqq \sum_{e\in X}c(e)\).
For $X\subseteq E$, write $\ell[X]\coloneqq \Set{\ell(e)}{e\in X}$ for the set of labels occurring on $X$, in contrast to the sum $\ell(X)$.

We use standard matroid terminology, following~\cite{oxley2011Matroid}.
A \emph{matroid} $M=(E,\mathcal{I})$ consists of a pair formed by a finite ground set $E$ and a set family $\mathcal{I} \subseteq 2^E$ over $E$ such that $\emptyset \in \mathcal{I}$, $X \subseteq Y \in \mathcal{I}$ implies $X \in \mathcal{I}$, and for $X,Y \in \mathcal{I}$ with $|X|<|Y|$, there exists $e \in Y\setminus X$ such that $X \cup \set{e} \in \mathcal{I}$.
A subset $X \subseteq E$ is called \emph{independent} if $X \in \mathcal{I}$ and \emph{dependent} otherwise.
A \emph{base} is a maximal independent set.
Every base has the same cardinality, called the \emph{rank} of $M$. A matroid is a \emph{block matroid} if its ground set can be partitioned into two bases.  
For bases \(A\) and \(B\), write \(d(A,B)\coloneqq |A\setminus B|=|B\setminus A|\).
We will need the following \emph{perfect matching lemma}.

\begin{lemma}[{perfect matching~\cite{brualdi1969Comments}}]\label{lem:perfect-matching}
    For any bases $A$ and $B$, there is a bijection $\pi:A\setminus B\to B\setminus A$ such that $A-x+\pi(x)$ is a base for every $x\in A\setminus B$.
\end{lemma}
\begin{lemma}[{multiple exchange~\cite{greene1973multiple}}]\label{lem:multiple_exchange}
    For every bases $A,B$ and every subset $X\subseteq A\setminus B$, there exists a subset $Y\subseteq B\setminus A$ such that both $(A\setminus X)\cup Y$ and $(B\setminus Y)\cup X$ are bases.
\end{lemma}

A \emph{circuit} is a minimal dependent set. For a base $A$ of $M$ and $e\in E\setminus A$, the \emph{fundamental circuit} of $e$ wrt $A$ is the unique circuit in $A\cup \{e\}$ which necessarily contains $e$. 
The \emph{rank function} $r\colon 2^E \to \Z$ of a matroid $M$ is defined by $r(X) \coloneqq \max\Set{|I|}{I \subseteq X,\, I \in \mathcal{I}}$ for $X \subseteq E$.
The \emph{closure} of $X\subseteq E$, denoted by $\cl_M(X)$, is the unique maximal set $Y\supseteq X$ with $r(Y)=r(X)$.
We use the following lemma on a base and a circuit.

\begin{lemma}\label{lemma:base-circuit}
    Let $M$ be a matroid, $B$ be a base, and $C$ be a circuit. Then, for every $b\in B\cap C$, there exists $a\in C\setminus B$ such that $B-b+a$ is a base.
\end{lemma}
\begin{proof}
    Suppose not. Since $B\setminus \set{b}$ is independent and no element in $C\setminus B$ can be added to $B\setminus \set{b}$ while maintaining independence, we have $C\setminus B\subseteq \cl_M(B\setminus \set{b})$. This implies 
    \[
    b\in \cl_M(C\setminus\{b\})
    =\cl_M\bigl((C\setminus B)\cup(C\cap(B\setminus\{b\}))\bigr)
    \subseteq \cl_M(B\setminus\{b\}),
    \]
    which contradicts the independence of $B$.
\end{proof}

For a matroid $M=(E,\mathcal{I})$ and $X \subseteq E$, the \emph{restriction} to $X$, the deletion of $X$, and the contraction of $X$ are matroids defined by $M\mathbin{|}X \coloneqq (X,\, \Set{I \subseteq X}{I \in \mathcal{I}})$, $M\backslash X \coloneqq (E\setminus X,\, \Set{I \in \mathcal{I}}{I \cap X = \emptyset})$, and $M/X \coloneqq (E\setminus X,\, \Set{I \subseteq E\setminus X}{I \cup B \in \mathcal{I}})$, respectively, where $B$ is any base of $M\mathbin{|}X$.
The \emph{direct sum} of matroids $M_1=(E_1,\mathcal{I}_1)$ and $M_2=(E_2,\mathcal{I}_2)$ with $E_1\cap E_2=\emptyset$ is $M_1\oplus M_2\coloneqq (E_1\cup E_2,\Set{I_1\cup I_2}{I_1\in\mathcal{I}_1,\ I_2\in\mathcal{I}_2})$.
A (\emph{connected}) \emph{component} of $M$ is an equivalence class under the equivalence relation $\sim$ on $E$ defined by $e\sim e'$ if either $e=e'$ or $M$ has a circuit containing $\{e,e'\}$.
If $E_1,\ldots,E_m$ are the components of $M$, then $M=(M\mathbin{|}E_1)\oplus\cdots\oplus (M\mathbin{|}E_m)$; in particular, every base of $M$ is the disjoint union of bases of $M\mathbin{|}E_1,\ldots,M\mathbin{|}E_m$.

Let $A$ be an arbitrary base of $M=(E,\mathcal{I})$, the \emph{exchangeability graph} $G(A)$ is the bipartite graph where the vertex bipartition is $(A, E\setminus A)$, and each $e \in A$ and $f \notin A$ are adjacent if $A-e+f$ is a base.
The connected components of $G(A)$ for an arbitrary base $A$ are precisely the components of $M$~\cite{krogdahl1977dependence}.

For a matroid $M=(E,\mathcal{I})$, a labeling function $\ell \colon E \to G$, and a finite set \(F\subseteq G\), a base \(B\) of $M$ is \emph{\(F\)-avoiding} if \(\ell(B)\notin F\).
Given a weight function $c\colon E \to \R$ in addition, a \emph{min-cost \(F\)-avoiding base} is one of minimum cost among all \(F\)-avoiding bases.
Let
\begin{align*}
    \ell(M) &\coloneqq \Set{\ell(B)}{\text{$B$ is a base of $M$}}, \\
    \ell_{\le \theta}(M) &\coloneqq \Set{\ell(B)}{\text{$B$ is a base of $M$ with $c(B)\le\theta$}}
\end{align*}
for $\theta \in \R$.

We use the following theorem from~\citep{scherk1955Distinct, kemperman1956Complexes, Kemperman1960SmallSumsets} (see also~\cite{lev2005restricted}).

\begin{theorem}[{Kemperman–Scherk~\citep{scherk1955Distinct, kemperman1956Complexes, Kemperman1960SmallSumsets}}]\label{thm:kemperman-scherk}
Let $G$ be an abelian group, $g\in G$ be a group element, and \(P,Q\subseteq G\) be finite nonempty subsets. If $g$ has a unique representation as \(p+q\) with \(p\in P\) and \(q\in Q\), then \(|P+Q|\ge |P|+|Q|-1\), where $P+Q \coloneqq \Set{p+q}{p \in P,\, q \in Q}$.
\end{theorem}

\section{Proximity Theorem}\label{sec:proximity}
This section is devoted to proving \cref{thm:proximity}.
Throughout this section, let $M=(E,\mathcal{I})$ be a matroid with rank function $r$, let $\ell\colon E\to G$ be a labeling function, let $c\colon E\to\R$ be a cost function, and suppose that $B$ is the unique min-cost $\ell(B)$-base.
Abusing notation, we also write $r$ for the rank of $M$ itself.
The case $r=0$ is immediate, so assume $r\ge1$ throughout this section. For $h\in G$, denote by $E_h\coloneqq\Set{e\in E}{\ell(e)=h}$.
\subsection{Counting colors in general matroids}\label{sec:coloful-bases}
In this subsection, we prove \cref{thm:weighted-counting}.
Denote by $B$ the unique min-cost $\ell(B)$-base throughout this subsection.

Order the elements in $B$ as $B=\{b_1,\ldots,b_r\}$ so that $c(b_1)\le\cdots\le c(b_r)$.
For $i\in[r+1]$, let $B_{<i}\coloneqq\{b_1,\ldots,b_{i-1}\}$, $B_{\ge i}\coloneqq B\setminus B_{<i}$, $M_i\coloneqq M/B_{<i}$, and $E_h^i\coloneqq E_h\setminus B_{<i}$ for $h\in G$.
For $i\in[r]$, let $B_{>i}\coloneqq B_{\ge i}\setminus\{b_i\}$ and define
\begin{align*}
R_h^i &\coloneqq \Set{e\in E_h^i}{c(e)\le c(b_i)}, \\
S_i   &\coloneqq \Set{\ell(A)}{A\text{ is a base of }M_i\text{ and }c(A)\le c(B_{\ge i})},\\
H_i   &\coloneqq \Set{h\in G}{b_i\in \cl_{M_i}(R_h^i)}.
\end{align*}
Finally, set $R_h^{r+1}\coloneqq E_h^{r+1}$ for $h\in G$ and $S_{r+1}\coloneqq\{0\}$.
In particular, $\ell(B_{\ge i})\in S_i$ and $\ell(b_i)\in H_i$, and then $g\coloneqq\ell(B_{\ge i})=\ell(B_{>i})+\ell(b_i)\in S_{i+1}+H_i$ for all $i\in [r]$. Below we show that such a representation of $g$ is unique.
\begin{lemma}\label{lemma:unique}
    For every $i\in [r]$, the element $g=\ell(B_{\ge i})$ has a unique representation as $g=s+h$ with $s\in S_{i+1}$ and $h\in H_i$. 
\end{lemma}
\begin{proof}
    Pick an arbitrary $s\in S_{i+1}$ and $h\in H_i$ such that $g=s+h$. 
    By the definition of $S_{i+1}$, there exists a base $A$ of $M_{i+1}=M_i/b_i$ with $\ell(A)=s$ and $c(A)\le c(B_{>i})$, which implies $A \cup \set{b_i}$ is a base of $M_i$. By the definition of $H_i$, we have $b_i\in \cl_{M_i}(R_h^i)$. Therefore, $R_h^i\not\subseteq \cl_{M_i}(A)$ since otherwise $b_i\in \cl_{M_i}(R_h^i)\subseteq \cl_{M_i}(A)$ which is a contradiction. Then, there exists $e\in R_h^i$ such that $A\cup \set{e}$ is a base of $M_i$, which means $B'\coloneqq B_{<i}\cup A\cup \set{e}$ is a base of $M$. Moreover, $$\ell(B')=\ell(B_{<i})+\ell(A)+\ell(e)=\ell(B_{<i})+s+h=\ell(B_{<i})+g,$$
    which is the same as the label of $B$ by the definition of $g$. Further,
    \[
    c(B')=c(B_{<i})+c(A)+c(e)\le c(B_{<i})+c(B_{>i})+c(b_i)=c(B).
    \]
    Therefore, by the uniqueness of $B$, we have $B'=B$, which implies $A=B_{>i}$ and $e=b_i$. In particular, $s=\ell(B_{>i})$ and $h=\ell(b_i)$. This completes the proof that the representation of $g$ is unique as $g=\ell(B_{>i})+\ell(b_i)$.
\end{proof}

We are ready to prove \Cref{thm:weighted-counting}, which is the main theorem of this section, restated below for readability.

\weightedcounting*
\begin{proof}
    For every $i\in[r]$, $s\in S_{i+1}$, and $h\in H_i$, there is a base $A$ of $M_{i+1}$ with $\ell(A)=s$ and $c(A)\le c(B_{>i})$, and there is $e\in R_h^i$ such that $A\cup\{e\}$ is a base of $M_i$.
    Moreover, $\ell(A\cup\{e\})=s+h$ and $c(A\cup\{e\})\le c(B_{>i})+c(b_i)=c(B_{\ge i})$, so $S_{i+1}+H_i\subseteq S_i$.
    Since $\ell(B_{>i})+\ell(b_i)$ has a unique representation in $S_{i+1}+H_i$ by \Cref{lemma:unique}, \Cref{thm:kemperman-scherk} gives
    \begin{equation}\label{eq:S-H}
        |S_i|\ge |S_{i+1}+H_i|\ge |S_{i+1}|+|H_i|-1.
    \end{equation}
    
    Denote by $r_i$ the rank function of matroid $M_i$. Next, we claim for $i \in [r]$ and $h \in G$ that
    \begin{equation}\label{eq:rank-reduction}
        r_i(R_h^i)-r_{i+1}(R_h^{i+1})\le \mathbf{1}\{h\in H_i\}.
    \end{equation}
    For $i<r$, the inequality $c(b_i)\le c(b_{i+1})$ gives $R_h^i\setminus\{b_i\}\subseteq R_h^{i+1}$, and this inclusion also holds for $i=r$ by the definition $R_h^{r+1}=E_h^{r+1}$.
    Hence $R_h^i\cup\{b_i\}\subseteq R_h^{i+1}\cup\{b_i\}$.
    Since $M_{i+1}=M_i/b_i$,
    \[
    \begin{aligned}
        r_{i+1}(R_h^{i+1})=r_{i}(R_h^{i+1}\cup \set{b_i})-1\ge r_{i}(R_h^i\cup \set{b_i})-1=\begin{cases}
            r_i(R_h^i)-1&\quad \text{if }b_i\in \cl_{M_i}(R_h^i),\\
            r_i(R_h^i)&\quad \text{if }b_i\notin \cl_{M_i}(R_h^i).
        \end{cases}
    \end{aligned}
    \]
    This implies \eqref{eq:rank-reduction} by the definition of $H_i$. Therefore,
    \begin{align*}
        |\ell_{\le c(B)}(M)|=|S_1|\ge&\sum_{i=1}^r (|H_i|-1)+|S_{r+1}|\\
        =&\sum_{i=1}^r \sum_{h\in G} \mathbf{1}\{h\in H_i\}-r+1\\
        \ge&\sum_{i=1}^r \sum_{h\in G}\bigl(r_i(R_h^i)-r_{i+1}(R_h^{i+1})\bigr)-r+1\\
        =&\sum_{h\in G} r(R_h)-r+1.\qedhere
    \end{align*}
    Here, the first inequality follows from \eqref{eq:S-H}. The second inequality follows from \eqref{eq:rank-reduction}. The last equality follows from the facts that $R_h^1=R_h$, $M_1=M$, and $M_{r+1}=M/B$ has rank zero.
\end{proof}
We state the following simple form of the unweighted corollary.
\begin{corollary}\label{cor:unweighted-counting}
    Let $M=(E,\mathcal{I})$ be a matroid with rank function $r$, $G$ be an abelian group, $\ell:E\rightarrow G$ be a labeling, and $c:E\rightarrow \R$ be a cost function. Suppose that $B$ is the $\ell(B)$-base of $M$. Then, $$|\ell(M)|\ge \sum_{h\in G} r(E_h)-r(E)+1.$$
\end{corollary}

\subsection{Counting colors in block matroids}
In this subsection, we prove a refined version of \cref{thm:weighted-counting}, namely \cref{thm:weighted-counting-block}.
Throughout this subsection, assume that $M=(E,\mathcal{I})$ is a block matroid of rank $r$, where $E$ can be partitioned into bases $E=A\mathbin{\dot\cup}B$ such that \(A\) is a min-cost base, and \(B\) is the unique min-cost $\ell(B)$-base.
We use the same notations as in \cref{sec:coloful-bases} for $b_i$, $E_h$, $M_i$, $R_h^i$, $S_i$, and $H_i$.

We start with several lemmas.

\begin{lemma}\label{lemma:circuit-max-cost}
    For every circuit $C$ of $M$, we have $\argmax\Set{c(e)}{e \in C} \cap B \ne \emptyset$.
\end{lemma}
\begin{proof}
    Suppose not. Then, every max-cost element of $C$ is in $C\cap A$. Pick an arbitrary max-cost element $a\in C\cap A$. Since $A$ is a base of $M$, it follows from \Cref{lemma:base-circuit} that there exists $b\in C\setminus A$ such that $A'\coloneqq A-a+b$ is a base. However, $c(b)<c(a)$, and thus $c(A')<c(A)$, which is a contradiction to the minimality of $A$.
\end{proof}

\begin{lemma}\label{lemma:independency}
    $E_h$ is independent for every $h \in G$.
\end{lemma}
\begin{proof}
    Suppose the statement is false for some $h\in G$. Then, there exists a circuit $C\subseteq E_h$. By \Cref{lemma:circuit-max-cost}, $C$ has a max-cost element $b\in C\cap B$. It follows from \Cref{lemma:base-circuit} that there exists $a\in C\setminus B$ such that $B'=B-b+a$ is a base. Moreover, since $a,b\in C\subseteq E_h$, $a$ and $b$ have the same label $h$, and thus $\ell(B')=\ell(B)$. Further, $c(B')\le c(B)$ as $b$ is a max-cost element in circuit $C$. This is a contradiction to the uniqueness of $B$.
\end{proof}

\begin{lemma}\label{lemma:closure}
    For every $i\in [r]$ and $h \in G$, $b_i\in \cl_{M_i}(E_h^i)$ if and only if $b_i\in \cl_{M_i}(R_h^i)$.
\end{lemma}
\begin{proof}
    The reverse implication is immediate because $R_h^i\subseteq E_h^i$.
    For the forward implication, suppose that $b_i\in\cl_{M_i}(E_h^i)$.
    If $b_i\in E_h^i$, then $b_i\in R_h^i$, and the conclusion is immediate.
    Otherwise, there is a circuit $C$ of $M_i$ with $b_i\in C\subseteq E_h^i\cup\{b_i\}$.
    We claim that every element of $C$ has cost at most $c(b_i)$.
    Suppose not; then the maximum cost of an element in $C$ is greater than $c(b_i)$.
    Lift $C$ to a circuit $C'\subseteq C\cup B_{<i}$ of $M$.
    Since $c(b)\le c(b_i)$ for every $b\in B_{<i}$, the max-cost element of circuit $C'$ belongs to $C$. Then, it follows from \Cref{lemma:circuit-max-cost} that one can choose a max-cost element $e$ of $C'$ so that $e\in C'\cap B_{>i}$.
    By \Cref{lemma:base-circuit}, there is $e'\in C'\setminus B$ such that $B-e+e'$ is a base.
    Since $e,e'\in C\setminus\{b_i\}\subseteq E_h^i$, this base has label $\ell(B)$ and cost at most $c(B)$, contradicting the uniqueness of $B$.
    Thus $C\setminus\{b_i\}\subseteq R_h^i$, and hence $b_i\in\cl_{M_i}(R_h^i)$.
\end{proof}

In a block matroid under our assumptions, we can prove a refined version of \Cref{thm:weighted-counting}.
\begin{theorem}\label{thm:weighted-counting-block}
    Let $M=(E,\mathcal{I})$ be a rank-$r$ matroid where $E=A\mathbin{\dot\cup} B$, $G$ be an abelian group, $\ell:E\rightarrow G$ be a labeling, and $c:E\rightarrow \mathbb{R}$ be a cost function. Suppose that $A$ is a min-cost base and $B$ is the unique min-cost $\ell(B)$-base of $M$. Then, $|\ell_{\le c(B)}(M)|\ge r+1.$
\end{theorem}
\begin{proof}
    Let $r_i$ be the rank function of $M_i$.
    We claim that for each $i\in [r]$ and $h \in G$,
    \begin{equation}\label{eq:rank-diff}
        r_i(E_h^{i})-r_{i+1}(E_h^{i+1})=\mathbf{1}\{h\in H_i\}.
    \end{equation}
    Indeed,
    \[
    \begin{aligned}
        r_{i+1}(E_h^{i+1})&=r_i(E_h^{i+1}\cup \set{b_i})-1=r_i(E_h^{i}\cup \set{b_i})-1
        =r_i(E_h^{i})-\mathbf{1}\{b_i\in \cl_{M_i}(E_h^i)\}\\
        &=r_i(E_h^{i})-\mathbf{1}\{b_i\in \cl_{M_i}(R_h^i)\}=r_i(E_h^{i})-\mathbf{1}\{h\in H_i\},
    \end{aligned}
    \]
    where we use \Cref{lemma:closure} for the penultimate equality.

    Therefore,
    \[
    \begin{aligned}
        |\ell_{\le c(B)}(M)|=|S_1|
        &\ge\sum_{i=1}^r (|H_i|-1)+|S_{r+1}|\\
        &=\sum_{i=1}^r \sum_{h\in G} \mathbf{1}\{h\in H_i\}-r+1\\
        &=\sum_{i=1}^r \sum_{h\in G}\bigl(r_i(E_h^i)-r_{i+1}(E_h^{i+1})\bigr)-r+1\\
        &=\sum_{h\in G} r(E_h)-r+1\\
        &=\sum_{h\in G} |E_h|-r+1\\
        &=2r-r+1\\
        &=r+1,
    \end{aligned}
    \]
    where the first inequality is from \eqref{eq:S-H}; the second equality  is \eqref{eq:rank-diff}; the fourth equality follows from \Cref{lemma:independency}. 
\end{proof}

\subsection{Proof of the proximity theorem}
We prove Proximity \Cref{thm:proximity}, restated below, following ideas from~\citep{liuXu2024Congruency}.

\proximity*
\begin{proof}
    We prove the theorem for matroid $M_0=(E_0,\mathcal{I}_0)$ and its min-cost base \(A_0\). If \(A_0\) is \(F\)-avoiding, it is a min-cost \(F\)-avoiding base, and we are done. Otherwise, among all min-cost \(F\)-avoiding bases choose \(B_0\) minimizing \(d(A_0,B_0)\). Delete elements not in $A_0\cup B_0$ and contract elements in $J\coloneqq A_0\cap B_0$. Denote the resulting matroid by $M=(E,\mathcal{I})$ and by $A\coloneqq A_0\setminus J$, $B\coloneqq B_0\setminus J$. Then $M$ is a block matroid whose ground set can be partitioned into bases $E=A \mathbin{\dot\cup} B$. Let $r\coloneqq |A|=|B|$ be the rank of $M$. Let $F'\coloneqq F-\ell(J)$ be the translated forbidden set.

    We claim that $A$ is a min-cost base and $B$ is the unique min-cost $F'$-avoiding base of $M$. Indeed, for every base $A'$ of $M$, $A_0'\coloneqq A'\cup J$ is a base of $M_0$, and then $c(A')=c(A'_0\setminus J)\ge c(A_0\setminus J)=c(A)$. For every $F'$-avoiding base $B'\neq B$ of $M$, $B_0'\coloneqq B'\cup J$ is a base of $M_0$. Moreover, $\ell(B_0')=\ell(B')+\ell(J)\notin F'+\ell(J)=F$. Since $B_0$ is a min-cost $F$-avoiding base in $M_0$, $c(B')=c(B'_0\setminus J)\ge c(B_0\setminus J)=c(B)$. Further, since $E=A\mathbin{\dot\cup} B$ and $B'\neq B$, we have $B'\cap A\neq\emptyset$, which implies
    $$d(A_0,B_0')=|A_0\setminus B_0'|=|A\setminus B'|<|A|=d(A_0,B_0).$$
    Therefore, it has to be $c(B')>c(B)$, since otherwise $B'_0$ is also a min-cost $F$-avoiding base in $M_0$, but its distance to $A_0$ is smaller, which is a contradiction to the choice of $B_0$. This completes the proof that $B$ is the unique min-cost $F'$-avoiding base in $M$. In particular, $B$ is the unique min-cost $\ell(B)$-base of $M$.

    Recall that $E_h=\Set{e\in E}{\ell(e)=h}$ for every $h\in G$. It follows from \Cref{lemma:independency} that $E_h$ is independent for every $h\in G$. Applying \Cref{thm:weighted-counting-block} to $M$ and the bases $A,B$, we obtain
    \[
    |\ell_{\le c(B)}(M)|\ge r+1.
    \]
    Since $B$ is the unique min-cost $F'$-avoiding base in $M$, all labels in $\ell_{\le c(B)}(M)$ but $\ell(B)$ are in $F'$. Therefore, $r+1\le |\ell_{\le c(B)}(M)|\le |F'|+1=|F|+1$. Then, $d(A_0,B_0)=d(A,B)=r\le |F|$, which is the desired bound.
\end{proof}

\section{FPT algorithms for finding \texorpdfstring{$F$}{F}-avoiding bases}\label{sec:fpt}

Throughout this section, we denote the size of $E$ by $n$ and the rank of $M$ by $r$. When $|F|$ is fixed, \cref{thm:proximity} gives a polynomial-time local-search algorithm for finding a min-cost $F$-avoiding base.
Namely, the algorithm finds a min-cost base $A$ by the greedy algorithm, enumerates all the bases $B$ with $d(A,B) \le |F|$, and outputs a min-cost $F$-avoiding base among them.
Since the number of bases $B$ with $d(A,B) \le |F|$ is at most $\sum_{k=0}^{|F|} \tbinom{|A|}{k}\tbinom{|E \setminus A|}{k} \le (|F|+1)\cdot n^{2|F|}$, the running time is polynomial when $|F|$ is fixed.
However, this enumeration algorithm is XP rather than FPT when parameterized by $|F|$.

In this section, we give FPT algorithms parameterized by $|F|$.
\cref{sec:signature-enumeration} presents a signature-enumeration subroutine common to the unweighted and weighted settings.
Using it, \cref{sec:unweighted-fpt} gives an FPT algorithm for finding an $F$-avoiding base in the unweighted case where $c \equiv 0$, and \cref{sec:weighted-fpt} extends it to the case of general $c$.

We assume that a matroid $M=(E,\mathcal{I})$ is given by an independence oracle, and independence queries and group operations in $G$ take unit time.
Let $T_{\mathrm{MI}}$ and $T_{\mathrm{WMI}}$ denote the running times of unweighted and weighted matroid intersection algorithms, respectively, for two rank-$r$ matroids over $n$ elements.
For example, one may take $T_{\mathrm{MI}}$ as $\widetilde O(nr^{5/6})$ and $\widetilde O(nr^{3/4})$ using the deterministic and randomized algorithms of Blikstad~\cite{blikstad2021Breaking}, respectively, and $T_{\mathrm{WMI}}=\widetilde O(nr^2)$ using the weighted matroid intersection algorithm of Brezovec, Cornu\'ejols, and Glover~\cite{brezovecCornuejolsGlover1986Two}.

\subsection{Signature enumeration}\label{sec:signature-enumeration}
The following lemma formalizes the signature-enumeration step in a form usable for our algorithms.
It restricts the exchanged elements to a prescribed label set $\Lambda\subseteq G$, so that the number of valid signatures is governed by $|\Lambda|$ and the proximity bound $d$. Recall that $\ell[X]=\Set{\ell(e)}{e\in X}$. 

\begin{lemma}\label{lem:fpt-profile}
    Let $A$ be a base, let $F,\Lambda \subseteq G$ be finite subsets, and let $d$ be a positive integer.
    There is an algorithm that finds a min-cost base $B$ among the ones satisfying
    \begin{equation}\label{eq:signature-Lambda}
        B \text{ is $F$-avoiding},\quad \ell[A \mathbin{\triangle} B]\subseteq\Lambda,\quad \text{and}\quad \|\sigma(A)-\sigma(B)\|_1 \le 2d,
    \end{equation}
    or correctly reports that no such base exists in $O\bigl(\tbinom{|\Lambda|+d}{d}^2\bigl(T_{\mathrm{WMI}}+|F|\bigr)\bigr)$ time.
    When $c\equiv0$, $T_{\mathrm{WMI}}$ can be replaced by $T_{\mathrm{MI}}$.
\end{lemma}


\begin{proof}
    For $g\in G$, let $E_g\coloneqq \Set{e\in E}{\ell(e)=g}$ and $A_g\coloneqq \Set{e\in A}{\ell(e)=g}$. Note that every base $B$ with $\ell[A\mathbin{\triangle}B]\subseteq\Lambda$ satisfies $\sigma_g(B)=\sigma_g(A)$ for all $g\in G\setminus\Lambda$.
    We may assume without loss of generality that $0<|\Lambda|\le n$ by restricting $\Lambda$ to $\ell[E]$ and by noting that the case where $\Lambda=\emptyset$ is trivial.

    Recall that $d$-close valid signatures $b\in\Z_{\ge0}^G$ are the ones satisfying \eqref{eq:candidate}. Enumerate all $d$-close valid signatures $b$ that additionally satisfy
    \begin{equation}\label{eq:candidate-signature}
        b_g=\sigma_g(A)\ \forall g\in G\setminus\Lambda.
    \end{equation}
    For each such $b$, form a partition matroid $N_b$ on $E$ whose partition classes and capacities are: part $E_\lambda$ with capacity $b_\lambda$ for all $\lambda\in \Lambda$, a single part $\bigcup_{g\in G\setminus \Lambda}A_g$ with capacity equal to its size, and the part of all remaining elements with capacity zero. Find a min-cost common base of $M$ and $N_b$, if one exists.
    Run the above weighted matroid intersection algorithm for $M$ and every $N_b$, and return the one with the smallest cost, or report that no base satisfying \eqref{eq:signature-Lambda} exists if every instance is infeasible.

    We show that the union of common bases for all instances is precisely the set of bases satisfying \eqref{eq:signature-Lambda}, which proves the correctness of the algorithm. A common base $B$ of $M$ and $N_b$ for an instance $b$ has $\sigma(B)=b$; hence $\|\sigma(A)-\sigma(B)\|_1\le2d$ and $\ell(B)=\sum_{h\in G}b_h\cdot h\notin F$ according to \eqref{eq:candidate}. Moreover,
    \[
    r=|B|\le \sum_{\lambda\in \Lambda} b_\lambda+ \Big|\bigcup_{g\in G\setminus \Lambda} A_{\lambda}\Big|
    =\sum_{\lambda\in \Lambda} b_\lambda+\sum_{g\in G\setminus A} \sigma_g(A)=\sum_{h\in G}b_h= r,
    \]
    where the first equality follows from the fact that $B$ is a base of $M$; the first inequality follows from the fact that $B$ is independent in $N_b$; the penultimate equality follows from \eqref{eq:candidate-signature}; the last equality follows from \eqref{eq:candidate}. Therefore, equality holds throughout, which implies $B$ saturates every part: it contains all elements of $\bigcup_{g\in G\setminus\Lambda}A_g$, exactly $b_\lambda$ elements of $E_\lambda$ for each $\lambda\in \Lambda$, and nothing else. It follows that $A\mathbin{\triangle}B\subseteq \bigcup_{\lambda\in\Lambda}E_\lambda$, which means $\ell[A\mathbin{\triangle}B]\subseteq\Lambda$.
    Conversely, for every base $B$ satisfying \eqref{eq:signature-Lambda}, its signature $b:=\sigma(B)$ is $d$-close and valid by definition, and satisfies \eqref{eq:candidate-signature} because $A$ and $B$ only differ by elements with labels in $\Lambda$.

    Next, we bound the number of signatures $b$ satisfying \eqref{eq:candidate} and \eqref{eq:candidate-signature}.
    Each $b$ decomposes uniquely as $b=\sigma(A)-u+v$, where $u,v\in\Z_{\ge0}^\Lambda$ have disjoint supports and satisfy $\sum_{\lambda \in \Lambda}u_\lambda=\sum_{\lambda \in \Lambda}v_\lambda=\frac12\|\sigma(A)-b\|_1\le d$ by $\sum_{h\in G} b_h=r=\sum_{h\in G}\sigma_h(A)$ and \eqref{eq:candidate-signature}. Since the number of vectors in $\{x\in \Z_{\ge 0}^{\Lambda}: \sum_{\lambda \in \Lambda}x_\lambda\le d\}$ is exactly $\tbinom{|\Lambda|+d}{d}$, the number of signatures in the enumeration is at most $\tbinom{|\Lambda|+d}{d}^2$.
    
    Finally, for each $b$, computing $\sum_{g\in G}b_g\cdot g=\ell(A)+\sum_{\lambda\in\Lambda}\bigl(b_\lambda-\sigma_\lambda(A)\bigr)\lambda$ takes $O(n)$ group operations by $|\Lambda|\le n$, 
    checking the $F$-avoiding condition and constructing the partition matroid from the precomputed classes take $O(n+|F|)$ time, and a min-cost common base is found in $T_{\mathrm{WMI}}$ time.
    As $T_{\mathrm{WMI}}\ge n$, the total running time is as claimed.
\end{proof}

The algorithm of Liu and Xu corresponds to applying \cref{lem:fpt-profile} with $\Lambda=G$ and $d=|G|-1$.
The dependence on $|G|$ of its running time thus arises from the dimension of the signatures.
Consequently, 
a na\"ive extension of this approach to an arbitrary forbidden label set $F \subseteq G$ yields an algorithm that is FPT parameterized by $|G|$, not $|F|$.
Our algorithms below instead construct label sets $\Lambda$ whose sizes are bounded by functions of $|F|$.

\subsection{Unweighted algorithm}\label{sec:unweighted-fpt}
Here, for the case where $c \equiv 0$, we provide an FPT algorithm for finding an $F$-avoiding base parameterized by $|F|$.
Our idea is to either find an $F$-avoiding base directly, or modify the labeling so that the number of appearing group elements can be bounded by a function of $|F|$.

\begin{theorem}\label{thm:fpt}
    Let $M=(E,\mathcal{I})$ be a matroid with rank function $r$, let $G$ be an abelian group, let $F\subseteq G$ be finite, and let $\ell:E\rightarrow G$ be a labeling.
    There is an $O\bigl({|F|}^2 n^2 + \tbinom{{|F|}^2}{|F|}^2 \bigl(T_{\mathrm{MI}}+|F|\bigr)\bigr)$ time algorithm that decides whether $M$ has an $F$-avoiding base, and if so, returns one.
\end{theorem}
\begin{proof}
    We may assume that $F \ne \emptyset$ and that $M$ is loopless, since the case $F=\emptyset$ is trivial and loops may be deleted. Define the set of difference of labels as $D \coloneqq  \Set{g-h}{g,h \in F}$. First, compute the components of $M$ by computing connected components of the exchangeability graph $G(A)$ for an arbitrary base $A$. Next, search for elements $e,e' \in E$ in the same component of $M$ such that $\ell(e)-\ell(e') \notin D$.
    If such elements exist, then $e \ne e'$ as $0 \in D$.
    Let $C$ be a circuit containing $\{e,e'\}$.
    As $C\setminus\{e'\}$ is independent, it extends to a base $A$, having $C$ as the fundamental circuit of $e'$.
    Hence, both $A$ and $B \coloneqq  A-e+e'$ are bases.
    Since $\ell(B)-\ell(A)=\ell(e')-\ell(e)\notin D$, at least one of $A$ and $B$ is $F$-avoiding.
    In an actual implementation, we apply a matroid intersection algorithm to $M/e\backslash e'$ and $M/e'\backslash e$.
    Since $e$ and $e'$ belong to the same component, these two matroids have a common base $I$.
    Then, $A \coloneqq  I+e$ is a base such that the associated fundamental circuit of $e'$ contains $e$.

    Suppose otherwise that any two elements in the same component have the label difference in $D$.
    Let $E_1,\dotsc,E_m$ be the components of $M$.
    Fix arbitrary $e_i \in E_i$ for $i \in [m]$ and define a new label function $\ell'\colon E \to G$ by $\ell'(e) \coloneqq  \ell(e) - \ell(e_i)$ for $e \in E_i$ and $i \in [m]$. Then, $\ell'$ only takes values in $D$. Moreover, 
    since every base $B$ of $M$ is the disjoint union of bases of $M\mathbin{|}E_i$ for $i \in [m]$, $\ell'(B)=\ell(B)-\sum_{i=1}^m r(E_i)\ell(e_i)$. Therefore, $B$ is $F$-avoiding under $\ell$ if and only if $B$ is $F'$-avoiding under $\ell'$, where $F' \coloneqq  \Set{g-\sum_{i=1}^m r(E_i) \ell(e_i)}{g \in F}$.

    The rest of the algorithm is the signature enumeration procedure from  \cref{sec:signature-enumeration}.
    Fix an arbitrary base $A$.
    By \cref{thm:proximity} with $c \equiv 0$, there is an $F'$-avoiding base $B$ under $\ell'$ with $d(A,B) \le |F'|=|F|$ if an $F'$-avoiding base exists.
    Run \cref{lem:fpt-profile} with the labeling $\ell'$, the base $A$, the forbidden set $F'$, the label set $\Lambda \coloneqq  D$, and the proximity bound $d \coloneqq  |F|$.
    Since $\ell'$ takes values in $D$, the label condition of \cref{lem:fpt-profile} is automatic, and hence \cref{lem:fpt-profile} finds a base if and only if an $F'$-avoiding base exists.
    If a base is found, return it, as it is $F$-avoiding under $\ell$; otherwise, report that no $F$-avoiding base exists.

    The running time is analyzed as follows.
    The components of $M$ can be found in $O(nr)$ time, and a pair $\{e,e'\}$ as above can be found by checking the label difference of $O(n^2)$ pairs which takes time $O(|D| n^2)$. 
    If such a pair $\{e,e'\}$ is found, we apply a matroid intersection algorithm once.
    Otherwise, \cref{lem:fpt-profile} takes $O\bigl(\tbinom{|D|+|F|}{|F|}^2\bigl(T_{\mathrm{MI}}+|F|\bigr)\bigr)$ time. By $|D| \le |F|^2$, we obtain the claimed complexity.
\end{proof}

\subsection{Weighted algorithm}\label{sec:weighted-fpt}

We now extend \cref{thm:fpt} to a general cost function $c$.

\begin{theorem}\label{thm:weighted-fpt}
    Let $M=(E,\mathcal{I})$ be a matroid with rank function $r$, let $G$ be an abelian group, let $F\subseteq G$ be finite, let $\ell:E\rightarrow G$ be a labeling, and let $c:E\rightarrow\R$ be a cost function. 
    There is an algorithm that either returns a min-cost $F$-avoiding base or reports that none exists in $2^{O(|F|^2)}\cdot O\bigl(n^2\,T_{\mathrm{WMI}}\bigr)$ time.
\end{theorem}

\subsubsection{Preparation}
Towards proving \cref{thm:fpt}, we prepare lemmas on basis exchangeability. We call a subset $Z\subseteq E$ an \emph{exchangeable set} for a base $A$ if $A\symdiff Z$ is a base. In particular, this requires that $X\coloneqq Z\cap A$ and $Y\coloneqq Z\setminus A$ have the same cardinality. The following lemma from~\cite{schrijver1979Linking} is originally stated in for linking systems (also known as bimatroids), and translates to the form following through the correspondence between linking systems and matroids with a distinguished base~\cite[Theorem~3.2]{schrijver1979Linking}. We include a proof for completeness.

\begin{lemma}[{Schrijver~\cite[Section~4, Lemma]{schrijver1979Linking}}]\label{lem:fpt-schrijver}
    Let $A$ be a base, and let $Z_1,Z_2\subseteq E$ be disjoint subsets.
    Suppose that $A-x+y$ is not a base for all $x\in Z_1\cap A$ and $y\in Z_2\setminus A$.
    Then, $Z_1\cup Z_2$ is exchangeable for $A$ if and only if both $Z_1$ and $Z_2$ are exchangeable for $A$.
\end{lemma}
\begin{proof}
Let $B\coloneqq A\symdiff (Z_1\cup Z_2)$, and let $X_i\coloneqq A\cap Z_i$ and $Y_i\coloneqq B\cap Z_i$ for $i\in \{1,2\}$. Then, $\{X_1,X_2\}$ and $\{Y_1,Y_2\}$ bipartitions of $A \setminus B$ and $B \setminus A$, respectively, because $Z_1\cap Z_2=\emptyset$. Let $G\coloneqq G(A)[A \mathbin{\triangle} B]$ be the exchangeability graph of $A$ restricted to $A\mathbin{\triangle} B$. By hypothesis, there are no edges between $X_1$ and $Y_2$.

Suppose first that $Z_1\cup Z_2$ is exchangeable for $A$, which means that $B$ is a base. By the multiple exchange \Cref{lem:multiple_exchange}, there exists $Y_0\subseteq B\setminus A$
such that both $(A\setminus X_1)\cup Y_0$ and $(B\setminus Y_0)\cup X_1$ are bases. By the perfect matching \Cref{lem:perfect-matching}, $G$ contains a perfect matching between $X_1$
and $Y_0$. Since there are no edges between $X_1$ and $Y_2$,
we have $Y_0\subseteq Y_1$. As $|Y_0|=|X_1|=|Y_1|$, it follows
that $Y_0=Y_1$. The two bases above are therefore
$(A\setminus X_1)\cup Y_1$ and $(A\setminus X_2)\cup Y_2$.

Conversely, suppose that 
$D_i\coloneqq A\symdiff Z_i=(A\setminus X_i)\cup Y_i$ is a base for each $i\in \{1,2\}$. Set $I\coloneqq A\cap B$. The set $I\cup Y_2$ is independent,
being a subset of $D_2$. Extend $I\cup Y_2$ to a base $D\subseteq I\cup Y_2\cup X_2\cup Y_1$ by adding elements from $D_1$. 
Write $D=(A\setminus X_0)\cup Y_0$, where
$X_1\subseteq X_0\subseteq X_1\cup X_2$ and $Y_2\subseteq Y_0\subseteq Y_1\cup Y_2$.
By \Cref{lem:perfect-matching}, there is a perfect matching in $G$ between $X_0$ and $Y_0$. Every vertex of $Y_2$ must be matched into $X_0\setminus X_1$,
so $|X_0\setminus X_1|\ge |Y_2|=|X_2|$. Consequently, $X_0=X_1\cup X_2$. Since $|Y_0|=|X_0|$, we also
have $Y_0=Y_1\cup Y_2$. Hence $D=B$, proving that $B$ is a base.
\end{proof}

Let $A$ be a base, let $Z\subseteq E$ be a subset such that $X\coloneqq Z\cap A$ and $Y\coloneqq Z\setminus A$ have the same cardinality, and let $\pi:X\rightarrow Y$ be a bijection. Let $P\coloneqq\{(x,\pi(x)):x\in X\}$ be the pairing between $X$ and $Y$ defined by the bijection. Define the auxiliary digraph $D^\pi(A)$ on the pairs $P$ as follows: place an arc $(e,e')$ for distinct pairs $e=(x,y),e'=(x',y')\in P$ if and only if $A-x'+y$ is a base. 
For a subset $Q\subseteq P$, denote by $X(Q)\coloneqq\{x:(x,\pi(x))\in Q\}$, $Y(Q)\coloneqq\{\pi(x):(x,\pi(x))\in Q\}$, and $Z(Q)\coloneqq X(Q)\cup Y(Q)$. The following lemma characterizes the exchangeability of $Z$ in terms of the strongly connected components of $D^\pi(A)$.


\begin{lemma}\label{lem:fpt-scc-factorization}
    Let $A$ be a base, let $Z\subseteq E$ be a subset such that $X\coloneqq Z\cap A$ and $Y\coloneqq Z\setminus A$ have the same cardinality, and let $\pi:X\rightarrow Y$ be a bijection. Then, $Z$ is exchangeable for $A$ if and only if $Z(Q)$ is exchangeable for $A$ for every strongly connected component $Q$ of $D^\pi(A)$.
\end{lemma}
\begin{proof}
    Let $P\coloneqq\{(x,\pi(x)):x\in X\}$ be the pairing defined by $\pi$; then $Z=Z(P)$. We proceed by induction on the number of strongly connected components
    of $D^\pi(A)$.
    The assertion is immediate when this number is at most one.
    Suppose that there are at least two strongly connected components.
    Let $Q$ be a sink strongly connected component, which means there is no outgoing arc from $Q$ to $P'\coloneqq P\setminus Q$.
    Equivalently,
    $A-x+y$ is not a base for any $x\in X(Q)$ and $y\in Y(P')$.
    Applying \cref{lem:fpt-schrijver} with
    $X_1\coloneqq X(P')$, $Y_1\coloneqq Y(P')$, $X_2\coloneqq X(Q)$, and $Y_2\coloneqq Y(Q)$, we see that
    $Z(P)$ is exchangeable if and only if both
    $Z(Q)$ and $Z(P')$ are exchangeable. Further, letting $\pi'$ be $\pi$ restricted to $P'$, we have $D^{\pi'}(A)=D^\pi(A)[P']$. Since $D^{\pi'}(A)$ has fewer strongly connected components, the induction hypothesis implies that $Z(P')$ is exchangeable if and only if $Z(Q')$ is exchangeable for every strongly connected component $Q'$ of $D^{\pi'}(A)$, which are precisely the strongly connected components of $D^{\pi}(A)$ contained in $P'$. Combining these equivalences proves the lemma.
\end{proof}

This lemma essentially establishes the equivalence of exchanging individual strongly connected components of $D^\pi(A)$ and the union of them. We remark that when the pairing $P=\{(x,\pi(x)):x\in A\}$ forms the unique perfect matching of the exchangeability graph of $A$, the digraph $D^\pi(A)$ is acyclic, and thus every strongly connected component is a singleton. Hence, \cref{lem:fpt-scc-factorization} specializes to the classical unique perfect matching lemma of Krogdahl~\cite{krogdahl1977dependence}.
The analogous factorization over the connected components of such a bipartite graph appears as Theorem~4.4 of~\cite{schrijver1979Linking}.

\subsubsection{Algorithm}

Throughout this and the following subsections, we assume that distinct bases have distinct costs.
To achieve this, fix an arbitrary ordering of $E$, and replace the cost $c(e)$ of each
element $e$ by the ordered pair $\widetilde c(e)\coloneqq (c(e),2^{p(e)})$, where $p(e)$ is the position of $e$ in the ordering.
Pairs are added coordinate-wise and compared lexicographically.
The min-cost base with respect to $\widetilde c$ in every nonempty family of bases is unique and is minimum with respect to $c$.
The same ordering of the bases can also be induced by the real-valued cost $c_\alpha(e)\coloneqq c(e)+\alpha 2^{p(e)}$ for sufficiently small $\alpha>0$.
Consequently, \cref{thm:proximity} applies to the tie-broken cost.
We suppress the tilde below. 
Under this assumption, the min-cost base of $M$, denoted by $A$ throughout, is unique.
Moreover, for every finite $F'\subseteq G$ for which an $F'$-avoiding base exists, the min-cost $F'$-avoiding base is unique.

The algorithm of \cref{thm:weighted-fpt} is the recursive procedure $\textsc{Solve}(F)$ shown in \cref{alg:solve}, invoked on the input forbidden set.
For a nonnegative integer $q\in \Z_{\ge 0}$, a \emph{$q$-step derivation} of $H \subseteq G$ is a subset $\Lambda \subseteq G$ consisting of the elements $\lambda$ that can be obtained from $H$ by a sequence of $q$ steps $\lambda_0,\lambda_1,\ldots,\lambda_q=\lambda$, for some $\lambda_0\in H$ and $\lambda_{i}=\lambda_{i-1}\pm p_i$ for some $p_i\in F-\ell(A)$ for every $i\in [q]$.

\begin{algorithm}[htbp]
    \caption{$\textsc{Solve}(F)$}\label{alg:solve}
    \begin{algorithmic}[1]
        \algtext*{EndIf}
        \algtext*{EndFor}
        \State $A \gets \argmin\Set{c(B')}{B' \text{ is a base of } M}$\linelabel{line:greedy}
        \If{$\ell(A)\notin F$}
            \State \Return $A$\linelabel{line:returnA}
        \EndIf
        \State $K \gets \emptyset$
        \For{$f \in F$}
            \If{$\textsc{Solve}(F\setminus\{f\})$ returns a base $B_f$}\linelabel{line:recursive}
                \State $K \gets K \cup (B_f \mathbin{\triangle} A)$\linelabel{line:KL}
            \EndIf
        \EndFor
        \State $\mathcal B\gets\Set{A-x+y}{x\in A,\ y\in E\setminus A,\ A-x+y\text{ is an $F$-avoiding base}}$\linelabel{line:swap}
        \For{$a \in A$}
        \State  $\Lambda_a\gets 2|F|$-step derivation of $\ell[K] \cup \{\ell(a)\}$\linelabel{line:derivation}
            \If{\cref{lem:fpt-profile} for $A$, $F$, $\Lambda_a$, $d=|F|$ returns a base $B'$}\linelabel{line:signature_enum}
                \State $\mathcal{B} \gets \mathcal{B} \cup \{B'\}$\linelabel{line:profile}
            \EndIf
        \EndFor
        \If{$\mathcal B\ne\emptyset$}
            \State \Return $\argmin\Set{c(B')}{B' \in \mathcal{B}}$ \linelabel{line:returnmin}
        \Else
            \State \Return no $F$-avoiding base exists\linelabel{line:returnno}
        \EndIf
    \end{algorithmic}
\end{algorithm}

\subsubsection{Analysis}
We now turn to proving the correctness of \cref{alg:solve}.
Every base collected in $\mathcal B$ is an $F$-avoiding base, by construction in \cref{line:swap} and by \cref{lem:fpt-profile} in \cref{line:profile}.
The remaining task is to show that $\mathcal B$ contains the min-cost $F$-avoiding base, which we establish in \crefrange{lem:fpt-low-arcs}{lem:fpt-complete} under the following setting.

Consider an invocation of $\textsc{Solve}(F)$ such that an $F$-avoiding base exists, but the unique min-cost base $A$ is not $F$-avoiding. Denote by $B$ the unique min-cost $F$-avoiding base throughout. 
Denote by $\Delta\coloneqq c(B)-c(A) > 0$. We have $1 \le d(A,B)\le|F|$ by $\ell(A) \in F$ and \cref{thm:proximity}.
Assume that every call in \cref{line:recursive} returns the min-cost $F\setminus\{f\}$-avoiding base $B_f$ whenever one exists.
Then, $d(A,B_f)\le|F|-1$ for every returned base $B_f$ by \cref{thm:proximity}.
If an $F$-avoiding base of the form $A-x+y$ exists, its cost is at least $c(B)$; if equality holds, $A-x+y=B$ by uniqueness and is already collected in \cref{line:swap}.
Hence we may assume that every such base has cost $c(A-x+y)>c(B)$.

Call a pair $(x,y)$ with $x\in A$ and $y\in E\setminus A$ \emph{exchangeable pair} if $\{x,y\}$ is an exchangeable set for $A$, i.e., $A-x+y$ is a base. It is a \emph{light exchangeable pair} if additionally $c(y)-c(x)\le\Delta$. For a light exchangeable pair $(x,y)$, the cost $c(A-x+y)\le c(A)+\Delta=c(B)$. By our standing assumption concerning single exchanges, $A-x+y$ is not $F$-avoiding, which means
$\ell(y)-\ell(x)\in F-\ell(A)$. Observe that if two elements are connected by at most $q$ light exchangeable pairs, then their labels can go from one to the other by a $q$-step derivation. By our standing assumption, $\ell(A)\in F$, and thus one can always add $0\in F-\ell(A)$ to the previous label in a derivation sequence. Therefore, a $q$-step derivation is a subset of a $(q+1)$-step derivation.

To prove the correctness of the algorithm, it suffices to show that for some $a\in A$, a $2|F|$-step derivation of $\ell[K]\cup \{\ell(a)\}$ contains $\ell[A\symdiff B]$, where $K$ is constructed in \cref{line:KL} of \cref{alg:solve}. Indeed, this would imply that in line \cref{line:signature_enum} of the algorithm, the base $B$ lies in the search range of \cref{lem:fpt-profile} and will be returned as the unique min-cost $F$-avoiding base. To achieve this, we crucially use the light exchangeable pairs to connect $K$ to $A\symdiff B$.

Let $G^\circ(A)$ be the exchangeability graph $G(A)$ of $A$ restricted to the light exchangeable pairs. We consider the graph restricted to $A\symdiff B$ and shorten it by $G^\circ\coloneqq G^\circ(A)[A\symdiff B]$. Let $Z_1,\ldots,Z_m$ be the vertex sets of the connected components of $G^\circ$. The proof discusses two cases: (1) every component $Z_k$ either intersects $K$ or is incident to $K$ by a light exchangeable pair; (2) not the first case. The first case is great because $K$ can go to any element in $Z_k$ by at most $|Z_k|$ steps and thus traverse $A\symdiff B=\bigcup_{k\in [m]}Z_k$ in at most $|A\symdiff B|=2d(A,B)\le 2|F|$ steps. Therefore, we mainly focus on the second case in the following. Below we show in \cref{lem:fpt-components} that the union of an arbitrary subset of the components $Z_1,\ldots, Z_k$ is exchangeable for $A$. To prove this we go to the auxiliary digraph $D^\pi(A)$ defined before \cref{lem:fpt-scc-factorization} for a bijection $\pi$ from the perfect matching \cref{lem:perfect-matching}, and show that every $Z_k$ is a union of strongly connected components of the auxiliary digraph. Combining with the additional assumption that there exists a component $Z_k$ that does not satisfy the condition in (1), we prove that $B=B_f\symdiff Z_k$ for some $f\in F$ (\cref{lem:fpt-separated}); in particular $A\symdiff B=(A\symdiff B_f)\symdiff Z_k\subseteq K\cup Z_k$. Thus, we only need to add to $K$ one additional element $a\in Z_k\cap A$ to reach $K\cup Z_k$ in at most $|Z_k|-1$ steps.

We formalize the argument below. First, we need the following lemma.


\begin{lemma}\label{lem:fpt-low-arcs}
    Let $X\subseteq A$, let $Y\subseteq E\setminus A$, and let $\pi:X\rightarrow Y$ be a bijection such that $c(\pi(x))\ge c(x)$ for every $x\in X$ and $c(Y)-c(X) \le \Delta$. 
    Then, every arc from $(x,y)$ to $(x',y')$ inside a strongly connected component of $D^\pi(A)$ gives a light exchangeable pair  $(x',y)$.
\end{lemma}
\begin{proof}
    By the strong connectivity, there is a simple directed path $(x',y')=(x_1,y_1)\to (x_2,y_2)\to\cdots\to (x_l,y_l)=(x,y)$ in $D^\pi(A)$ such that $y_i=\pi(x_i)$ for every $i\in [l]$, which together with the arc $(x,y)\to (x',y')$ forms a directed cycle. By the definition of $D^\pi (A)$, $A-x_{i+1}+y_i$ is a base for every $i\in [l]$, where $l+1\coloneqq 1$. This implies $c(y_{i})-c(x_{i+1})\ge 0$ for every $i\in [l]$.
    Moreover, 
    \[
        \sum_{i=1}^{l} (c(y_{i})-c(x_{i+1}))=\sum_{i=1}^{l}(c(y_{i})-c(x_{i})) \le \sum_{i\in X}(c(\pi(x))-c(x))=c(Y)-c(X)\le \Delta,
    \]
    where the first inequality follows from the fact that $c(\pi(x))\ge c(x)$ for every $x\in X$. 
    As all terms in the sum are nonnegative, we have $c(y_{i})-c(x_{i+1}) \le \Delta$ for each $i \in [l]$. In particular, $c(y_{l})-c(x_{1})=c(y)-c(x')\le\Delta$.
    As $A-x'+y$ is a base, $(x',y)$ is a light exchangeable pair.
\end{proof}


\begin{lemma}\label{lem:fpt-components}
    Let $Z_1,\ldots,Z_m$ be the connected components of $G^\circ$. Then, for every $S\subseteq[m]$, $\bigcup_{k\in S}Z_k$ is an exchangeable set for $A$. In particular, for every $k\in [m]$, $Z_k$ is an exchangeable set for both $A$ and $B$.
\end{lemma}

\begin{proof}
Let $Z\coloneqq A\symdiff B$ and note that $Z_1,\ldots, Z_k$ partitions $Z$. Denote by $X\coloneqq Z\cap A$ and $Y\coloneqq Z\cap B$.
Let $\pi\colon X \to Y$ be a bijection such that $A-x+\pi(x)$ is a base for every $x\in X$, which exists by \cref{lem:perfect-matching}.
    Then $c(\pi(x))\ge c(x)$ for each $x \in X$, as $A$ is the unique min-cost base.
    Moreover, since $\sum_{x\in X} (c(\pi(x))-c(x)) = c(Y)-c(X)=c(B)-x(A)= \Delta$, the pair $(x,\pi(x))$ is light for every $x\in X$. In particular, $x$ and $\pi(x)$ lie in the same connected component of $G^\circ$. 
    By \cref{lem:fpt-low-arcs}, every arc $(x,\pi(x))\to (x',\pi(x'))$ inside each strongly connected component $Q$ of $D^\pi(A)$ gives a light exchangeable pair that joins $\{x,\pi(x)\}$ and $\{x',\pi(x')\}$.
    Consequently, $Z(Q)=\bigcup_{(x,\pi(x))\in Q}\{x,\pi(x)\}$ lies in a connected component of $G^\circ$. Therefore, every connected component $Z_k$ is a union of strongly connected components of $D^\pi(A)$. Since $B=A\symdiff Z$ is a base, i.e., $Z$ is exchangeable for $A$, \cref{lem:fpt-scc-factorization} implies that $Z(Q)$ is exchangeable for every strongly connected component $Q$. Consider an arbitrary subset $S\subseteq [m]$, and then $\bigcup_{k\in S}Z_k$ is also a union of strongly connected components of $D^\pi(A)$. Apply \cref{lem:fpt-scc-factorization} again with $\pi'$ being $\pi$ restricted to $\bigcup_{k\in S}Z_k$, we conclude that $\bigcup_{k\in S}Z_k$ is exchangeable because every strongly connected component of $D^{\pi'}(A)$ is a strongly connected component of $D^{\pi}(A)$ and thus exchangeable. Finally, $Z_k$ is exchangeable for $A$ by taking $S=\{k\}$ and exchangeable for $B$ by taking $S=[m]\setminus \{k\}$.
\end{proof}

Recall that the set $K$ constructed in \cref{line:KL} of \cref{alg:solve} is the union of $A\mathbin{\triangle} B_f$ for bases $B_f$ returned in \cref{line:recursive}.

\begin{lemma}\label{lem:fpt-separated}
    Suppose $Z$ is a connected component of $G^\circ$ such that it neither intersects $K$ nor is incident to $K$ by a light exchangeable pair.
    Then, for some $f\in F$, the base $B_f$ returned in \cref{line:recursive} satisfies $B_f\symdiff Z=B$.
\end{lemma}
\begin{proof}
    Denote by $X\coloneqq A\cap Z$ and $Y\coloneqq B\cap Z$; then the sets $A' \coloneqq A\symdiff Z= (A\setminus X)\cup Y$ and $B' \coloneqq B\symdiff Z= (B\setminus Y)\cup X$ are both bases by \cref{lem:fpt-components}. 
    As $A$ is the unique min-cost base, we have $c(A')>c(A)$, i.e., $c(Y)>c(X)$.
    Thus, $c(B')<c(B)$, implying that $B'$ is not $F$-avoiding by the fact that $B$ is the min-cost $F$-avoiding base.
    Let $f\coloneqq \ell(B') \in F$.
    As $B'$ avoids $F\setminus\{f\}$, the call $\textsc{Solve}(F\setminus\{f\})$ in \cref{line:recursive} returns a base $B_f$ with $c(B_f)\le c(B')<c(B)$, which implies $B_f$ is not $F$-avoiding. However, $B_f$ avoids $F\setminus\{f\}$, which forces $\ell(B_f)=f$. Denote by $Z'\coloneqq A\symdiff B_f$, and by $X'\coloneqq A\cap Z'$ and $Y'\coloneqq B\cap Z'$. Since $Z'\subseteq K$ by the definition of $K$, the assumption that $Z$ does not intersect $K$ implies that $Z\cap Z'=\emptyset$. Then, as $B_f=A\symdiff Z'$, it holds that $X\subseteq B_f$ and $Y\cap B_f=\emptyset$. Hence, $B_f'\coloneqq B_f\symdiff Z$ satisfies
    \begin{align*}
        c(B_f')&=c(B_f)+c(Y)-c(X) \le c(B') +c(Y)-c(X) = c(B), \\
        \ell(B_f')&=\ell(B_f)+\ell(Y)-\ell(X)=\ell(B')+\ell(Y)-\ell(X)=\ell(B)\notin F.
    \end{align*}

    It remains to show that $B_f'$ is a base, which would imply $B_f'=B$ by the uniqueness, as needed.
    By \cref{lem:perfect-matching} applied to the pairs of bases $(A,A')$ and $(A,B_f)$, there exist bijections $\pi_1\colon X\to Y$ and $\pi_2\colon X'\to Y'$ such that $A-x+\pi_1(x)$ and $A-x'+\pi_2(x')$ are bases for every $x\in X$ and $x'\in X'$. We define $\pi\colon X\cup X'\to Y\cup Y'$ as the union of $\pi_1$ and $\pi_2$, which is a well-defined bijection as $Z\cap Z'=\emptyset$. 
    Consider the auxiliary digraph $D^\pi(A)$ on pairs $P\coloneqq \{(x,\pi(x)):x\in X\cup X'\}$.
    
    We claim that $P_1\coloneqq \{(x,\pi_1(x)):x\in X\}$ and $P_2\coloneqq \{(x,\pi_2(x)):x\in X'\}$ are both unions of strongly connected components of $D^\pi(A)$. Suppose not, which means there exists a strongly connected component $Q\subseteq P$ that intersects both $P_1$ and $P_2$. In particular, $Q$ contains an arc from some  $(x,y)\in P_1$ to some $(x',y')\in P_2$, which corresponds to a basis exchange $A-x'+y$. Below we show that $(x',y)$ is a light exchangeable pair, which contradicts the assumption, as it would be a light exchangeable pair joining $K\supseteq Z'$ and $Z$. To this end, note that $c(\pi(x))>c(x)$ for every $x\in X$ because $A-x+\pi(x)$ is a base whose cost is larger than $A$. Moreover, $c(Y\cup Y')-c(X\cup X')=c(B_f')-c(A)\le c(B)-c(A)=\Delta$. It follows from \Cref{lem:fpt-low-arcs} that $(x',y)$ is a light exchangeable pair, as needed.

    Finally, apply \Cref{lem:fpt-scc-factorization} to $\pi_i$ for both $i\in \{1,2\}$, we obtain that every strongly connected component of $D^\pi(A)$ is exchangeable for $A$. Therefore, applying \Cref{lem:fpt-scc-factorization} again to $\pi$ proves that $Z\cup Z'$ is exchangeable, i.e., $B'_f=B_f\symdiff Z=A\symdiff Z'\symdiff Z$ is a base. As $B_f'$ is an $F$-avoiding base of cost at most $c(B)$, the minimality and uniqueness of $B$ give $B=B_f'$.
\end{proof}

\begin{lemma}\label{lem:fpt-complete}
    \cref{line:profile} adds the min-cost $F$-avoiding base $B$ to $\mathcal B$.
\end{lemma}
\begin{proof}
    It suffices to show that there exists $a \in A$, so that the $2|F|$-step derivation $\Lambda_a$ of $\ell[K]\cup\{\ell(a)\}$ satisfies $\ell[A\mathbin{\triangle}B]\subseteq\Lambda_a$. Denote by $\Lambda$ the $2|F|$-step derivation of $\ell[K]$, and then $\Lambda\subseteq \Lambda_a$ for every $a\in A$.
    Recall that by our standing assumption, $\ell(y)-\ell(x)\in F-\ell(A)$ for every light exchangeable pair $(x,y)$.

    We claim that for two vertices $u,v$ in $G^\circ$, if $G^\circ$ contains a path of length at most $q$ from $u$ to $v$, then $\ell(v)$ is in the $q$-step derivation of $\{\ell(u)\}$. Indeed, let $u=u_0,u_1,\ldots,u_l=v$ be the path for some $l\le q$. For each $i\in [l]$, either $(u_{i-1},u_{i})$ or $(u_{i},u_{i-1})$ is a light exchangeable pair, depending on which endpoint is in $A$. Then, either $\ell(u_{i})-\ell(u_{i-1})\in F-\ell(A)$ or $\ell(u_{i-1})-\ell(u_{i})\in F-\ell(A)$. Picking $p_i=\ell(u_{i})-\ell(u_{i-1})$ or $\ell(u_{i-1})-\ell(u_{i})$ so that $p_i\in F-\ell(A)$, we obtain that $\ell(u_i)=\ell(u_{i-1})\pm p_i$, and thus $\ell(v)=\ell(u_l)$ is in the $l$-step derivation of $\ell(u)=\ell(u_0)$ by definition. Since $0\in F-\ell(A)$ by assumption, $\ell(v)$ is in the $q$-step derivation of $\ell(u)$ by adding additional $0$'s to the remaining $q-l$ steps.


    Let $Z_1,\ldots,Z_m$ be the connected components of $G^\circ$. Suppose first that, for every $k \in [m]$, $Z_k$ intersects $K$ or is incident to $K$ by a light exchangeable pair.
    By $\ell[A\mathbin{\triangle}B]=\bigcup_{k\in[m]}\ell[Z_k]$, it suffices to show that $\Lambda$ contains $\ell[Z_k]$ for every $k\in[m]$. By the claim above, it suffices to show that for every $k\in [m]$, every vertex in $Z_k$ is connected to $K$ by a path of length at most $2|F|$. This follows from the fact that the path has length at most $|Z_k|\le \sum_{k=1}^m |Z_k|=|A\symdiff B|=2d(A,B)\le 2|F|$ by \cref{thm:proximity}. 

    Consider the remaining case where some $Z_k$ neither intersects $K$ nor is incident to $K$ by a light exchangeable pair.
    By \cref{lem:fpt-separated}, we have $B=B_f \symdiff Z_k$ for some $f \in F$.
    Then, $A\mathbin{\triangle}B=(A\mathbin{\triangle}B_f)\symdiff Z_k$, and hence $\ell[A\mathbin{\triangle}B]\subseteq\ell[K]\cup\ell[Z_k]$ by $A\mathbin{\triangle}B_f\subseteq K$.
    Picking a vertex $a\in A\cap Z_k$, every vertex in $K\cup Z_k$ can be reached from $K\cup a$ by a path of length at most $|Z_k|-1\le 2|F|$. By the claim, the $2|F|$-step derivation $\Lambda_a$ contains $\ell[K]\cup\ell[Z_k]\supseteq\ell[A\mathbin{\triangle}B]$.
\end{proof}

We are now ready to prove \cref{thm:weighted-fpt}.

\begin{proof}[Proof of \cref{thm:weighted-fpt}]
    We first show by induction on $|F|$ that $\textsc{Solve}(F)$ returns the min-cost $F$-avoiding base or correctly reports that none exists.
    If $\ell(A)\notin F$, then $A$ is the min-cost $F$-avoiding base and \cref{line:returnA} is correct; this covers $F=\emptyset$.
    Hence assume $\ell(A)\in F$.
    If no $F$-avoiding base exists, then $\mathcal B=\emptyset$ since every base collected in $\mathcal B$ is $F$-avoiding, and the report in \cref{line:returnno} is correct.
    Otherwise, let $B$ be the min-cost $F$-avoiding base.
    By the induction hypothesis, every call in \cref{line:recursive} returns the min-cost $(F\setminus\{f\})$-avoiding base whenever one exists.
    If some $F$-avoiding base of the form $A-x+y$ has cost $c(B)$, then it equals $B$ and $B\in\mathcal B$ by \cref{line:swap}; otherwise the setting of \crefrange{lem:fpt-low-arcs}{lem:fpt-complete} is realized, and $B\in\mathcal B$ by \cref{lem:fpt-complete}.
    In both cases, since every base in $\mathcal B$ is $F$-avoiding and hence has cost at least $c(B)$, \cref{line:returnmin} returns $B$.

    We then analyze the running time; recall that each cost operation takes $O(n)$ time under the perturbation.
    Consider one invocation of \cref{alg:solve} with $F\ne\emptyset$, excluding the recursive calls in \cref{line:recursive}.
    \cref{line:greedy,line:KL,line:swap} take $O(n^3+|F|n^2+{|F|}^2)$ time.
    Since $|A\symdiff B|=2d(A,B_f)\le2(|F|-1)$ for every base $B_f$ returned in \cref{line:recursive}, we have $|K|\le2|F|(|F|-1)$ and hence $|\ell[K]\cup\{\ell(a)\}|\le2{|F|}^2$.
    By commutativity, each element of a $2|F|$-step derivation is determined by an initial group element and a multiset of at most $2|F|$ signed elements of $F-\ell(A)$. Thus, for every $a\in A$, the $2|F|$-step derivation $\Lambda_a$ of $\ell[K] \cup \set{\ell(a)}$ satisfies
    \[
        |\Lambda_a|\le2{|F|}^2\binom{2|F|+2|F|}{2|F|}\le2{|F|}^2\,2^{4|F|}.
    \]
    Each $\Lambda_a$ in \cref{line:derivation} can be constructed in $O\bigl(|F|^4\,2^{8|F|}\bigr)$ time by enumerating these multisets and removing duplicate labels. Apply \cref{lem:fpt-profile} with $d=|F|$ and since
    \[
        \binom{|\Lambda_a|+|F|}{|F|}^2\le(|\Lambda_a|+|F|+1)^{|F|}\le(2|F|+2)^{4|F|}\,2^{8|F|^2}=O(|F|^{4|F|}2^{9|F|^2}),
    \]
    it takes $O\bigl(|F|^{4|F|}\,2^{9|F|^2}\bigl(n\,T_{\mathrm{WMI}}+|F|\bigr)\bigr)$, where the factor $n$ accounts for additions and comparisons of the symbolically perturbed weights.
    We call \cref{lem:fpt-profile} for every $a\in A$, so this cost is multiplied by $|A| \le n$.
    Selecting a min-cost base in \cref{line:returnmin} takes $O(n)$ time per collected base.
    As $T_{\mathrm{WMI}}\ge n$, one invocation thus takes $O\bigl(|F|^{4|F|}\,2^{9|F|^2}\,n^2\,T_{\mathrm{WMI}}\bigr)$ time.
    The recursion has at most $\sum_{j=0}^{|F|}|F|!/j!\le3\,|F|!$ invocations.
    The total running time is therefore $O\bigl(|F|!\,|F|^{4|F|}\,2^{9|F|^2}\,n^2\,T_{\mathrm{WMI}}\bigr)=2^{O(|F|^2)}\cdot O\bigl(n^2\,T_{\mathrm{WMI}}\bigr)$, since $O(|F|!\,|F|^{4|F|})=2^{O(|F|\log |F|)}$.
\end{proof}

\section*{Acknowledgments}
Siyue Liu was supported in part by the Air Force Office of Scientific Research under agreement number FA9550261B210. Taihei Oki was supported by JST FOREST Grant Number JPMJFR232L, JST BOOST Grant Number JPMJBY25A6, and JSPS KAKENHI Grant Number JP18J22141. We thank Avinash Bhardwaj for sharing the progress on representable matroids \cite{bhardwaj2026Sharp}, and acknowledge that both results were discovered independently.

\paragraph{AI disclosure.}

The authors acknowledge the use of GPT-5.6 Sol in the discovery and exploration of the proofs. All arguments were independently verified by the authors. The manuscript was written entirely by the authors, who take full responsibility for its content. 

\bibliographystyle{alpha}
\bibliography{references}

\end{document}